\documentclass{amsart}
\usepackage{amsmath, amsthm, amssymb, mathrsfs}
\usepackage{url}
\usepackage{setspace}
\usepackage{url}
\usepackage{wasysym}
\allowdisplaybreaks

\renewcommand\theenumi{\roman{enumi}}
\newtheoremstyle{citing}
{}{}
{\itshape}
{}{\bfseries}
{.}
{ }
{\thmnote{#3}}
\theoremstyle{plain}
\newtheorem{thm}{Theorem}[section]
\newtheorem{corol}[thm]{Corollary}
\newtheorem{lemma}[thm]{Lemma}
\newtheorem{prop}[thm]{Proposition}

\newtheorem{remark}[thm]{Remark}

\theoremstyle{definition}
\newtheorem{defn}[thm]{Definition}

\theoremstyle{citing}

\newcommand{\vectwo}[2]{\left[\begin{array}{c} #1 \\ #2 \end{array}\right]}
\newcommand{\inv}[1]{\frac{1}{ #1 }}
\newcommand{\n}{\|}
\newcommand{\maxsym}{\vee}
\newcommand{\minsym}{\wedge}
\newcommand{\E}{\mathbb{E}}
\newcommand{\R}{\mathbb{R}}
\newcommand{\N}{\mathbb{N}}

\renewcommand{\P}{\mathbb{P}}
\newcommand{\mL}{\mathcal{L}}
\newcommand{\F}{\mathscr{F}}

\title{Vector-valued stochastic delay equations - a semigroup approach}
\author{Sonja Cox and Mariusz G\'orajski}
\address{Sonja Cox, Delft Institute of Applied Mathematics, Delft University of Technology, P.O. Box 5031, 2600 GA Delft, The Netherlands; e-mail: \url{s.g.cox@tudelft.nl}}
\address{Mariusz G\'orajski,
Faculty of Mathematics and Computer Science
University of  \L \'od\'z, Banacha 22, 90-238 \L \'od\'z, Poland; Department of Econometrics,
University of  \L \'od\'z,
Rewolucji 1905 r. 41/43, 90-214 \L \'od\'z, Poland; e-mail: \url{mariuszg@math.uni.lodz.pl}}
\begin{document}
\begin{abstract}
Let $E$ be a type 2 \textsc{umd} Banach space, $H$ a Hilbert space and let $p\in [1,\infty)$. Consider the following stochastic delay equation in $E$:
\begin{align}\label{SDE1}
\left\{\begin{array}{rll} 
dX(t)&= AX(t) + C X_t + B(X(t),X_t)dW_H(t),& t>0;\\
X(0)&=x_0;\\
X_0&=f_0, 
\end{array}\right. \tag{SDE}
\end{align}
where $A:D(A)\subset E\rightarrow E$ is the generator of a $C_0$-semigroup. The operator $C\in \mL(W^{1,p}(-1,0;E),E)$ is given by a Riemann-Stieltjes integral, and $B:E\times L^p(-1,0;E)\rightarrow \gamma(H,E)$ is a Lipschitz function. Moreover $W_H$ is an $H$-cylindrical Brownian motion adapted to $(\F_t)_{t\geq 0}$ and $x_0\in L^2(\F_0,E)$, $f_0\in L^2(\F_0, L^p(-1,0;E))$. We prove that a solution to \eqref{SDE1} is equivalent to a solution to the corresponding stochastic Cauchy problem, and use this to prove the existence, uniqueness and continuity of a solution to \eqref{SDE1}.
\end{abstract}
\maketitle

\section{Introduction} 
Let $E$ be a type 2 \textsc{umd} Banach space and let $H$ be a Hilbert space. Consider the following stochastic delay equation in $E$:
\begin{align}\label{SDE}
\left\{\begin{array}{rll} 
dX(t)&= AX(t) + C X_t + B(X(t),X_t)dW_H(t),& t>0;\\
X(0)&=x_0;\\
X_0&=f_0, 
\end{array}\right. \tag{SDE}
\end{align}
where for a strongly measurable function $x:[-1,\infty)\rightarrow E$ and $t\geq 0$ we define $x_t:[-1,0]\rightarrow E$ by
\begin{align*}
x_t(s):= x(t+s), \quad s\in [-1,0].
\end{align*}
We assume that $A:D(A)\subset E\rightarrow E$ is closed, densely defined and linear, and generates a $C_0$-semigroup. Define $\mathcal{E}^p(E):=E\times L^p(-1,0;E)$. We assume that $C\in \mL(W^{1,p}(-1,0;E),E)$ for some $p\in [1,\infty)$, and that $B:\mathcal{E}^p(E)\rightarrow \gamma(H,E)$ is a Lipschitz function. Here $\gamma(H,E)$ is the space of $\gamma$-radonifying operators from $H$ to $E$, see Section \ref{s:stochint} below. Moreover, $W_H$ is an $H$-cylindrical Brownian motion on a given probability space $(\Omega, (\F_t)_{t\geq 0},\F,\P)$. The initial value $[x_0, f_0]$ is assumed to be in $L^2(\F_0,\mathcal{E}^p(E))$.\par
Recall that \textsc{umd} stands for unconditional martingale difference sequences; the class of \textsc{umd} Banach spaces includes Hilbert spaces and $L^p$ spaces for $p\in (1,\infty)$. The type of a Banach space is defined in terms of randomized sequences; see Section \ref{s:SDE} below. We note that Hilbert spaces have type $2$ and $L^p$-spaces with $p\in [1,\infty)$ have type $\min\{p,2\}$.\par
We follow the semigroup approach to the delay equation as given in the monograph of Batkai and Piazzera \cite{batkaiPiazzera}. This forces us to assume in addition that $C$ is given by the Riemann-Stieltjes integral
\begin{align*}
Cf:= \int_{-1}^{0} fd\eta,
\end{align*}
where $\eta: [-1,0]\rightarrow \mathcal{L}(E)$ is of bounded variation. This defines an element of $\mL(W^{1,p}(-1,0;E),E)$ by the Sobolev embedding. (One may allow for more general $C\in \mL(W^{1,p}(-1,0;E),E)$; it suffices for $C$ to satisfy the conditions of Theorem 3.26 in \cite{batkaiPiazzera}. The Riemann-Stieltjes integral is the most important example of such a $C$.)\par

One can define a closed operator $\mathcal{A}$ on $\mathcal{E}^p(E)$ by 
\begin{align}\label{opA}
D(\mathcal{A})&=\{[x,f]\in D(A)\times W^{1,p}(-1,0;E)\,:\, f(0)=x\}; \notag \\
\mathcal{A}&=\left[\begin{array}{cc} A & C \\ 0 & \frac{d}{dt} \end{array}\right].
\end{align}
This operator generates a $C_0$-semigroup $(\mathcal{T}(t))_{t\geq 0}$ on $\mathcal{E}^p(E)$ (see \cite{batkaiPiazzera}, Theorem 3.29)
and the stochastic delay equation can be rewritten as a stochastic Cauchy problem in $\mathcal{E}^p(E)$ given by
\begin{align}\label{SDCP}
\left\{ \begin{array}{rll} dY(t)&=\mathcal{A}Y(t)dt + \mathcal{B}(Y(t))dW_H(t), &t\geq 0;\\
Y(0)&=\vectwo{x_0}{f_0}, \end{array}\right. \tag{SDCP}
\end{align}
where $\mathcal{B}(Y(t)):=[B(Y(t)),0]^T$.
\par
The approach we take is to prove existence, uniqueness and continuity of a solution to the stochastic Cauchy problem \eqref{SDCP} and then translate these results to corresponding results for the stochastic delay equation \eqref{SDE}. The monograph by Da Prato and Zabczyk \cite{DaPratoZabczyk} gives an extensive treatment of the stochastic Cauchy problem in Hilbert spaces. The stochastic Cauchy problem in Banach spaces has been considered in the work by Brze\'zniak \cite{brze95} and Van Neerven, Veraar and Weis \cite{vanNeervenVeraarWeis_SEEinUMD}, however, they both consider the case that $\mathcal{A}$ generates an analytic semigroup. Nevertheless their approach is a valuable starting point for studying \eqref{SDCP}.\par
Following the approach of the above mentioned authors we consider the following variation of constants formula:
\begin{align}\label{voc}
Y(t) = \mathcal{T}(t)Y(0) + \int_{0}^{t}\mathcal{T}(t-s)\mathcal{B}(Y(s))dW_H(s),
\end{align}
where the precise definition of the stochastic integral above and the relevant theory on vector-valued stochastic integrals will be given in Section
\ref{s:stochint}. A process satisfying \eqref{voc} is usually referred to as a \emph{mild solution}. The existence of a mild solution to \eqref{SDCP} is proved
by a fixed-point argument (see Section \ref{s:SDE}, Theorem \ref{t:exist}). Using the factorization method we prove the continuity of a mild solution to
\eqref{SDCP} (Theorem \ref{t:regSDCP}). In Section \ref{s:SCP} we give general conditions under which a mild solution is equivalent to what we call a
\emph{generalized strong solution} of the stochastic Cauchy problem. Finally, Theorem \ref{t:repr} states that solutions to \eqref{SDCP} and \eqref{SDE} are
equivalent, which is proved by using the concept of a generalized strong solution. Combining all these results we obtain existence, uniqueness and continuity of a solution to \eqref{SDE}, see Corollaries \ref{c:SDEexist} and \ref{c:SDEcont}.\par
An obvious consequence of our results is that one has the existence of a solution for initial value $f_0\in L^2(\F_0,L^1(-1,0;E))$. The $L^1$-norm is a
natural choice in population dynamics, see \cite[Example 3.16]{batkaiPiazzera}. The equivalence of solutions to \eqref{SDE} and to \eqref{SDCP} is useful
because the latter can be studied in the framework of the stochastic abstract Cauchy problem; thus answering questions concerning e.g.\ regularity and
invariant measures of the solutions to \eqref{SDE} (see Theorem \ref{t:regSDCP} and Remark \ref{r:invmeasure}). We also have that the solution to \eqref{SDCP}
is a Markov process, whereas the solution to \eqref{SDE} is not.\par
For the theory of stochastic delay equations in the case that $E$ is finite-di\-men\-sional we refer to the monographs by Mohammed \cite{Moh:84} and Mao
\cite{Mao:97} and references therein. In particular we wish to mention \cite{chojnowskaMichalik_reprSDE}, where equivalence of solutions to the stochastic delay
equation and the corresponding abstract Cauchy problem has been shown by Chojnowska-Michalik for the Hilbert space case, i.e.\ the case that $p=2$ and $E$ is
finite-dimensional. Similar results concerning the abstract Cauchy problem arising from delay equations with state space $C([0,1])$ with additive noise are
given by Van Neerven and Riedle \cite{vanNeervenRiedle_SDE}. For a general class of spaces including the $\mathcal{E}^p$-spaces the variation of constants
formula for finite-dimensional delay equations with additive noise and a bounded delay operator is discussed in Riedle \cite{riedle_SDE}. The latter articles
both consider the stochastic convolution as a stochastic integral in a locally convex space. So far there is no suitable interpretation for the stochastic
integral of a stochastic process in a locally convex space, hence this approach fails for equations with multiplicative noise.\par
Stochastic delay equations where $E$ is a Hilbert space and $p=2$ have been considered by Taniguchi, Liu, and Truman \cite{TanLiuTru:02}, Liu \cite{Liu:08} and
Bierkens, Van Gaans and Verduyn-Lunel \cite{BierGaansVer09}. Both \cite{TanLiuTru:02} and \cite{Liu:08} prove existence and uniqueness of solutions to
\eqref{SDE}; in \cite{TanLiuTru:02} it is assumed that $A$ generates an analytic semigroup, whereas in \cite{Liu:08} the noise is assumed to be additive. In
\cite{BierGaansVer09} the existence of an invariant measure has been studied. Very recently, Crewe \cite{Crewe:10} has taken it upon himself to prove existence,
uniqueness and regularity properties of \eqref{SDE} in \textsc{umd} Banach spaces under the assumption that $A$ generates an analytic semigroup.
\section{Preliminaries: Stochastic integration in Banach spaces}\label{s:stochint}
In this section we briefly recall some theory for stochastic integration in Banach spaces as introduced in \cite{vanNeervenVeraarWeis}. Throughout this section
let $H,\mathcal{H}$ denote Hilbert spaces and let $F$ denote a Banach space. By $L^0(\Omega;F)$ we denote the complete metric space of strongly measurable
functions on $\Omega$ with values in $F$ equipped with the topology of convergence in probability.\par
To build stochastic integrals of $\mL(H,F)$-valued processes we start by considering finite rank adapted step processes, i.e.\ processes of the form
\begin{align*}
\Phi(t,\omega)= \sum_{n=1}^{N}1_{(t_{n-1},t_n]}(t)
\sum_{m=1}^{M} 1_{A_{nm}}(\omega) \sum_{k=1}^{K}h_k \otimes x_{nmk},
\end{align*}
where $0=t_0<t_1<...<t_N$, $A_{nm}\in \F_{t_{n-1}}$, $x_{nmk}\in F$ and $(h_k)_{k\geq 1}$ is an orthonormal system in $H$. If $W_H$ is an $H$-cylindrical
Brownian motion adapted to $(\F_t)_{t\geq 0}$, then the integral of $\Phi$ with respect to $W_H$ is given by:
\begin{align*}
\int_0^{t_N} \Phi dW_H &=\sum_{n=1}^{N} \sum_{m=1}^{M} 1_{A_{nm}} \sum_{k=1}^{K} (W_H(t_n)h_k-W_H(t_{n-1})h_k) x_{nmk}.
\end{align*}
To extend this to general processes, we need some extra terminology:
\begin{defn}
Let $(\Omega,\F,P)$ be a probability space with filtration $(\F_t)_{t\geq 0}$. A process $\Phi:[0,\infty)\times \Omega\rightarrow \mL(H,F)$ is called
$H$-\emph{strongly measurable} if for every $h\in H$ the process $\Phi h$ is strongly measurable. The process is called \emph{adapted} if $\Phi h$ is adapted
for each $h\in H$ and we say that $\Phi$ is \emph{scalarly in} $L^q(\Omega;L^2(0,\infty;H))$ for some $q\in [0,\infty]$ if for all $x^*\in F^*$ one has
$\Phi^*x^*\in L^q(\Omega;L^2(0,\infty;H))$.
\end{defn}
The stochastic integral for general $\mL(H,F)$-valued processes is defined as follows: 
\begin{defn}\label{d:Lp_stochInt}
Let $W_H$ be an $H$-cylindrical Brownian motion. An $H$-strongly measurable adapted process $\Phi:[0,t]\times \Omega\rightarrow \mL(H,F)$ is called
\emph{stochastically integrable with respect to $W_H$} if there exists a sequence of finite rank adapted step processes $\Phi_n:[0,t]\times \Omega\rightarrow \mL(H,F)$
such that:
\begin{enumerate}
\item for all $h\in H$ and $x^*\in F^*$ we have $\lim_{n \rightarrow \infty}\langle \Phi_n h, x^* \rangle = \langle \Phi h, x^*\rangle$ in measure on
$[0,t]\times \Omega$;
\item there exists a process $X\in L^0(\Omega;C([0,t];F))$ such that 
\begin{align*}
\lim_{n\rightarrow \infty}\int_{0}^{\cdot} \Phi_n dW_H & = X \quad \textrm{in probability}.
\end{align*}
\end{enumerate}
The stochastic integral of $\Phi$ is then defined as
\begin{align*}
\int_{0}^{\cdot} \Phi dW_H := X.
\end{align*}
\end{defn}
A characterization of the processes which are stochastically integrable is obtained by means of the \emph{$\gamma$-radonifying norm}. Let $(\gamma_j)_{j\geq 1}$
be a sequence of independent standard Gaussian random variables. A bounded operator $R$ from $\mathcal{H}$ to $F$
is called \emph{$\gamma$-summing} if 
$$ \n R \n_{\gamma_{\infty}(\mathcal{H},F)}^2 := \sup_h \E\ \Big\n \sum_{j=1}^k \gamma_j R h_j \Big\n_F^2$$
is finite, where the supremum is taken over all finite orthonormal systems 
$h=(h_j)_{j=1}^k$ in $\mathcal{H}$. It can be shown that $\n\cdot \n_{\gamma_{\infty}(\mathcal{H},F)}$ is indeed a norm under which the space of $\gamma$-summing
operators is complete. We will later take $\mathcal{H}=L^2(0,t;H)$.\par
The space $\gamma(\mathcal{H},F)$ of $\gamma$-\emph{radonifying} operators is defined to be the closure of the finite rank operators under the norm
$\n\cdot\n_{\gamma_{\infty}}$; it is a closed subspace of $\gamma_{\infty}(\mathcal{H},F)$. Thus if $R\in \gamma(\mathcal{H},F)$ then range$(R)$ is separable
and there exists a separable subspace $\mathcal{H}_0\subset \mathcal{H}$ such that $R|_{\mathcal{H}_0^{\perp}} \equiv 0$. \par
A celebrated result of Kwapie\'n and Hoffmann-J{\o}rgensen \cite{HofJor, Kwa} implies that if $F$ does not contain a closed subspace isomorphic to $c_0$ then
$\gamma(\mathcal{H},F)=\gamma_\infty(\mathcal{H},F)$. This is the case for the spaces $L^p(-1,0;F)$ if $p\in [1,\infty)$ and $F$ is a \textsc{umd} Banach space.
\par
Note also that every $\gamma$-radonifying operator is compact and that the class of $\gamma$-radonifying operators is a left- and right ideal in the set of
bounded operators:
\begin{align*}
\n SRT \n_{\gamma(\mathcal{H}_1,F_2)}& \leq \n S\n_{\mL(F_1,F_2)}\n R \n_{\gamma(\mathcal{H}_2,F_1)} \n T \n_{\mL(\mathcal{H}_1,\mathcal{H}_2)}, 
\end{align*}
where $\mathcal{H}_1,\mathcal{H}_2$ are Hilbert spaces and $F_1,F_2$ are Banach spaces.\par
In what follows we will use the notation $A\lesssim_p B$ to express the fact that there exists a constant $C>0$, depending on $p$, such that $A\leq CB$. We
write $A\eqsim_p B$ if $A\lesssim_p B\lesssim_p A$. \par
Theorem 5.9 and Theorem 5.12 in \cite{vanNeervenVeraarWeis} state the relation between the $\gamma$-rado\-ni\-fying norm and the stochastically integrable
processes (see also \cite{coxveraar} for relation \eqref{BDG}). We summarize these results as follows:
\begin{thm}\label{t:stochint}
Let $F$ be a \textsc{umd} space. For an $H$-strongly measurable adapted process $\Phi:[0,t]\times \Omega \rightarrow \mL(H,F)$ belonging to
$L^0(\Omega;L^2(0,t;H))$ scalarly, the following are equivalent:
\begin{enumerate}
\item $\Phi$ is stochastically integrable;\label{l:sintble}
\item there exists a process $\eta\in L^0(\Omega;C([0,t];F))$ such that for all $x^*\in F^*$ we have
\begin{align*}
\langle \eta, x^* \rangle &=\int_{0}^{\cdot}\Phi^*(s)x^* dW_H(s) \quad \textrm{a.s.};
\end{align*} \label{l:sweak}
\item $\Phi$ represents an element $R_{\Phi}\in L^{0}(\Omega; \gamma(L^2(0,t;H),F))$ which is defined as follows: \label{l:gamma}
\begin{align*}
R_{\Phi}(\omega)f &:= \int_{0}^{t} \Phi(s,\omega)f(s)\ ds
\end{align*}
($f\in L^2(0,t;H)$).
\end{enumerate}
In this situation one has $\eta=\int_0^{\cdot}\Phi dW_H$ and for all $p \in (0,\infty)$
\begin{align}\label{BDG}
\E\ \sup_{0\leq s\leq t} \Big\n \int_{0}^{s} \Phi(u) dW_H(u) \Big\n_F^p \eqsim_p \E\ \n R_{\Phi} \n_{\gamma(L^2(0,t;H),F)}^{p}.
\end{align}
\end{thm}
\begin{remark}\label{r:measrep}
If $\Phi$ is $H$-strongly measurable and $R_{\Phi}\in \gamma(L^2(0,t;H),F)$ a.s.\ then by \cite[Lemma 2.5, 2.7 and Remark 2.8]{vanNeervenVeraarWeis} one
automatically obtains that $R_{\Phi}\in L^{0}(\Omega; \gamma(L^2(0,t;H);F))$. Thus in this situation one may assume without loss of generality that $H$ and $F$
are separable.\par
\end{remark}
From now on we shall simply write $\n \Phi \n_{\gamma(0,t;H,F)}$ to denote the $\gamma(L^2(0,t;H),F)$-norm of the operator $R_{\Phi}$ associated with $\Phi$.
\begin{remark}\label{r:stochint_umdmin}
One checks that if \eqref{l:gamma} in the theorem above holds, then $\Phi$ must be scalarly in $L^0(\Omega;L^2(0,t;H))$. Moreover, the implication
\eqref{l:sintble}$\Rightarrow$\eqref{l:sweak} holds for arbitrary Banach spaces. This follows from the Burkholder-Davis-Gundy inequalities (see the proof of
Theorem 3.6 in \cite{vanNeervenVeraarWeis}).\par
For $1<p<\infty$ one has that if $F$ is a \textsc{umd} space then so is $\mathcal{E}^p(F)$. However, $L^1$ is not a \textsc{umd} space so neither is
$\mathcal{E}^1(-1,0;F)$. Fortunately, $L^1$ does have the (weaker) decoupling property as introduced by Kwapie{\'n} and Woyczy{\'n}ski in \cite{KwWo1}. If $F$
has the decoupling property then $\mathcal{E}^p(F)$ is Banach space satisfying the decoupling inequality (by a Fubini argument, see \cite{coxveraar}). It was
proved in \cite{coxveraar} that for spaces with the decoupling property implication \eqref{l:gamma}$\Rightarrow$\eqref{l:sintble} remains valid. The two-sided
estimate as given in \eqref{BDG} need not hold in such spaces, but it is shown in \cite{coxveraar} that in spaces with the decoupling property the following
one-sided estimate holds for all $p\in(0,\infty)$:
\begin{align}\label{BDG2}
\E\ \sup_{0\leq s\leq t} \Big\n \int_{0}^{s} \Phi(u) dW_H(u) \Big\n_F^p \lesssim_p \E\ \n \Phi \n_{\gamma(0,t;H,F)}^{p}.
\end{align}
Note that in particular the integral process $t\mapsto \int_{0}^{t} \Phi(s)dW_H(s)$ is continuous.
\end{remark}
For $\gamma$-radonifying operators with values in an $L^p$-space we have the following isomorphism (see \cite{vanNeervenVeraarWeis_SEEinUMD}, Proposition 2.6):
\begin{lemma}\label{l:LpGammaFub}
Let $(S,\mathcal{S},\mu)$ be a $\sigma$-finite measure space, let $\mathcal{H}$ be a Hilbert space and let $p\in[1,\infty)$. Then the mapping
$U:L^p(S;\gamma(\mathcal{H},F))\rightarrow \mathcal{L}(\mathcal{H},L^p(S;F))$ defined by
\begin{align*}
((Uf)h)(\xi):=f(\xi)h,\quad \xi \in S, h\in \mathcal{H},
\end{align*}
defines an isomorphism $U$ of $L^p(S;\gamma(\mathcal{H},F))$ onto $\gamma(\mathcal{H},L^p(S;F))$.
\end{lemma}
The following stochastic Fubini theorem is based on \cite[Theorem 3.5]{NeerVer}.
%
\begin{lemma}\label{l:stochFubini}
Let $(S,\mathcal{S},\mu)$ be a $\sigma$-finite measure space and let $F$ be a Banach space satisfying the decoupling property. Let $\Phi:S\times [0,t] \times
\Omega \rightarrow \mL(H,F)$ and for $s\in S$ define $\Phi_s:[0,t]\times\Omega\rightarrow \mL(H,F)$ by $\Phi_s(u,\omega)=\Phi(s,u,\omega)$. Assume the following
is satisfied:
\begin{enumerate}
\item \label{meascond1} $\Phi$ is $H$-strongly measurable;
\item \label{meascond2} For all $s\in S$ and all $h\in H$ the section $\Phi_s h$ is progressive;
\item \label{intcond1} For almost all $u\in [0,t]$ and almost all $\omega\in \Omega$ one has $\Phi(\cdot, u,\omega)h \in L^1(S;F)$ for all $h\in H$ and the
operator $\int_S \Phi d\mu :H\rightarrow F$ defined by $\int_S \Phi d\mu h := \int_S \Phi h d\mu$ is in $\mL(H,F)$;
\item \label{intcond3} The process $u\mapsto\int_S \Phi(s,u) d\mu(s)$ represents an element of $\gamma(0,t;H,F)$ a.s.; 
\item \label{intcond4} The function $s\mapsto \Phi_s$ represents an element of $L^1(S;\gamma(0,t;H,F))$ a.s.
\end{enumerate}
Then the function $s\mapsto \int_{0}^{t} \Phi(s,u) dW_H(u)$ belongs to $L^1(S;F)$ a.s.\ and
\begin{align}\label{fub}
\int_S \int_{0}^{t} \Phi dW_H d\mu &= \int_{0}^{t}\int_S  \Phi d\mu dW_H \quad \textrm{a.s.}
\end{align}
\end{lemma}
\begin{proof}
Due to condition \eqref{intcond4} and the Fubini isomorphism in Lemma \ref{l:LpGammaFub} one has that $\Phi$ represents an element of $\gamma(0,t;H,L^1(S;F))$
a.s. As $\Phi$ is assumed to be $H$-strongly measurable we may assume $H$ and $F$ to be separable by Remark \ref{r:measrep}. This implies that $\Phi^*x^*$ is
strongly measurable for all $x^*\in F^*$ by Pettis's measurability theorem, and that $\Phi_s^*x^*$ is progressive for all $x^*\in F^*$, all $s\in S$ and all
$h\in H$.\par
Moreover, because $\Phi$ represents an element of $\gamma(0,t;H,L^1(S;F))$ a.s., the process $\Psi:[0,t]\times \Omega\rightarrow \mL(H,L^1(S;F))$ defined by
$$\Psi(u,\omega)(s):= \Phi(s,u,\omega)$$ is stochastically integrable, and by arguments similar to those in the proof of \cite[Theorem 3.5]{NeerVer} it follows
that
$$ \int_{0}^{t} \Phi(s,u) dW_H(u) = \Big(\int_{0}^{t} \Psi(u)dW_H(u)\Big) (s) \quad \textrm{a.s. for almost all }s\in S.$$ This proves that the integral with
respect to $\mu$ on the left-hand side of \eqref{fub} is well-defined. \par
Condition \eqref{intcond1} implies that the process in condition \eqref{intcond3} is well-defined, and this condition in combination with Theorem
\ref{t:stochint} and Remark \ref{r:stochint_umdmin} implies that the stochastic integral on the right-hand side of \eqref{fub} is well-defined.\par
Fix $x^*\in F^*$, then $\Phi^*x^*:S\times [0,t] \times \Omega \rightarrow H$ satisfies conditions (i)-(iii) of \cite[Theorem 3.5]{NeerVer} and hence by that
Theorem we have:
\begin{align*}
\int_S \int_{0}^{t} \Phi^*x^* dW_H d\mu &= \int_{0}^{t}\int_S  \Phi^*x^* d\mu dW_H \quad \textrm{a.s.}
\end{align*}
Although the null-set on which the above fails may depend on $x^*$, this suffices due to the fact that $F^*$ is weak$^*$-separable.
\end{proof}
In the next section we will need the following lemma which shows that as in the case of the Bochner integral, a closed operator can be taken out of a stochastic
integral.
\begin{lemma}\label{l:Aint=intA}
let $F$ be a Banach space satisfying the decoupling property and let $A:D(A)\subset F \rightarrow F$ be a closed, densely defined operator. Suppose
$\Phi:[0,t]\times \Omega\rightarrow \mL(H,F)$ is an $H$-strongly measurable adapted process that represents an element of $\gamma(0,t;H,F)$ a.s. Suppose that
one has $\Phi(s)h\in D(A)$ for all $s\in(0,t)$ and all $h\in H$ a.s., where the null sets are independent of $h$. Suppose moreover that $A\Phi$ is again an
$H$-strongly measurable adapted process that represents an element of $\gamma(0,t;H,F)$ a.s. Then $\int_{0}^{t} \Phi dW_H \in D(A)$ a.s.\ and
\begin{align*}
A\int_{0}^{t} \Phi dW_H &= \int_{0}^{t} A\Phi dW_H \quad \textrm{a.s.} 
\end{align*}
\end{lemma}
\begin{proof}
Define random variables $\eta:=\int_0^t \Phi dW_H$ and $\zeta:=\int_0^t A\Phi dW_H$ and observe that by implication (iii)$\implies$(ii) in Theorem
\ref{t:stochint}, which holds for Banach spaces with decoupling property, one has that for all $x^*\in F^*$:
\begin{align*}
\langle \eta, x^*\rangle &= \int_{0}^{t} \Phi^*(s)x^* dW_H(s) \quad \textrm{a.s.,} \\
\langle \zeta, x^*\rangle &= \int_{0}^{t} (A\Phi(s))^*x^* dW_H(s) \quad \textrm{a.s.}
\end{align*}
In particular for $x^*\in D(A^*)$ one has $(A\Phi(s))^*x^*=\Phi^*(s)A^*x^*$, and thus for such $x^*$ one has:
\begin{align}
\langle (\eta,\zeta), (-Ax^*,x^*)\rangle &= \langle \eta, -A^*x^* \rangle + \langle \zeta, x^*\rangle =0\quad a.s. \label{Aint_h1}
\end{align}
Note that the null-set on which the equation above fails to hold may depend on $x^*$. However, as $\Phi$ and $A\Phi$ are assumed to be $H$-strongly measurable
and in $\gamma(0,t;H,F)$ a.s.\ we may assume $F$ to be separable by Remark \ref{r:measrep}. Hence $(F\times F)/ \mathcal{G}(A)$ is separable, where
$\mathcal{G}(A)$ is the graph of $A$, and thus by Hahn-Banach there exists a countable subset of $((F\times F)/ \mathcal{G}(A))^*= \mathcal{G}(A)^{\bot}$ that
separates the points of $(F\times F)/\mathcal{G}(A)$. 

Moreover, one checks that if $(x_1^*,x_2^*) \in \mathcal{G}(A)^{\bot}$ then $x_2^*\in D(A^*)$ and $x_1^*=-A^*x_2^*$. Thus there exists a sequence
$(-Ax^*_n,x^*_n)_{n\in\N}$ that separates points in  $(F\times F)/\mathcal{G}(A)$. As equation \eqref{Aint_h1} holds for arbitrary $x^*\in D(A^*)$, it holds
simultaneously for all $x_n^*$, on a set of measure one. Therefore $(\eta, \zeta) \in \mathcal{G}(A)$, i.e.\ $\eta \in D(A)$ and $A\eta=\zeta$ a.s.
\end{proof}
\section{The Stochastic Cauchy Problem}\label{s:SCP}
In the introduction we mentioned that the stochastic delay equation \eqref{SDE} can be rewritten as a stochastic Cauchy problem. In this section we briefly
consider the stochastic Cauchy problem in general. Let $F$ be a Banach space with the decoupling property and $H$ a Hilbert space, and let $A:D(A)\subset
F\rightarrow F$ be the generator of a $C_0$-semigroup $(T(t))_{t\geq0}$ on $F$. Let $W_H$ be an $H$-cylindrical Brownian motion and let $B:F\rightarrow
\mL(H,F)$ be continuous (where $\mL(H,F)$ is endowed with the strong operator topology). We consider the following problem:
\begin{align}\label{SCP}
\left\{ \begin{array}{rll} dY(t)&=AY(t)dt + B(Y(t))dW_H(t), &t\geq 0;\\
Y(0)&=Y_0. \end{array}\right. \tag{SCP}
\end{align}
\begin{defn}\label{d:strongsolSCP}
An $H$-strongly measurable adapted process $Y$ is called a \emph{generalized strong solution} to \eqref{SCP} if $Y$ is a.s.\ locally Bochner integrable and for
all $t> 0$:
\begin{enumerate}
\item $\int_{0}^{t} Y(s)ds \in D(A)$ a.s.,
\item $B(Y)$ is stochastically integrable on $[0,t]$,
\end{enumerate}
and
$$Y(t) - Y_0 = A \int_{0}^{t} Y(s)ds + \int_{0}^{t} B(Y(s)) dW_H(s)\quad a.s.$$
\end{defn}
We use the term `generalized strong solution' to distinguish this solution concept from the conventional definition of a `strong solution', which concerns a
process satisfying $Y(t)\in D(A)$ a.s.\ for all $t\geq 0$ (see \cite{DaPratoZabczyk}). This assumption is not suitable for our situation, see Remark
\ref{r:strongDPZ} below.
\begin{thm}\label{t:varcons}
Let $Y$ be an $F$-valued $H$-strongly measurable adapted process. For $t\geq 0$ define $\int_{0}^{t} T(s)B(Y(u))ds\in \mL(H,F)$ by $$\Big(\int_{0}^{t} T(s)B(Y(u))
ds\Big)h:=\int_{0}^{t} T(s)B(Y(u))h ds.$$ Assume that for all $t> 0$ the following processes are in $\gamma(0,t;H,F)$ a.s.:
\renewcommand\theenumi{\alph{enumi}}
\begin{enumerate}
\item $B(Y)$;\label{si1}
\item $u\mapsto T(t-u)B(Y(u))$;\label{si2}
\item $u\mapsto \int_{0}^{t-u}T(s)B(Y(u))ds$;\label{si4}
\end{enumerate}
and that for all $t>0$
\begin{equation}\label{si5}
    \int_0^t\n T(s-\cdot)B(Y(\cdot))\n_{\gamma(0,s,H;F)}ds<\infty.
\end{equation}
Then $Y$ is a generalized strong solution to \eqref{SCP} if and only if $Y$ satisfies, for all $t\geq 0$,
\begin{align}\label{varcon}
Y(t) &= T(t)Y_0 + \int_{0}^{t}T(t-s)B(Y(s))dW_H(s) \quad \textrm{a.s.}
\end{align}
\end{thm}
\renewcommand\theenumi{\arabic{enumi}}
\begin{remark}\label{r:varcons}
\begin{enumerate}
\item If $Y$ is strongly measurable and adapted then the processes in \eqref{si1},  \eqref{si2} and  \eqref{si4} are $H$-strongly measurable and adapted.
\item \label{Bgamma} If $B:F\rightarrow \gamma(H,F)$ then for all $u\in[0,t]$ almost all paths $s\mapsto T(s)B(Y(u))$ are locally Bochner integrable in
$\gamma(H,F)$ because $B(Y(u))$ is the limit of finite-rank operators in $\gamma(H,F)$.
\item If $F$ is a \textsc{umd} space and $(T(s))_{0\leq s\leq t}$ is $\gamma$-bounded for all $t>0$ then \eqref{si2} and \eqref{si4} follow from \eqref{si1}.
For definition and details on $\gamma$-boundedness see \cite{vanNeervenVeraarWeis_SEEinUMD}, analytic semigroups are a typical example of $\gamma$-bounded
semigroups.
\end{enumerate} 
\end{remark}
\renewcommand\theenumi{\roman{enumi}}
\begin{proof}[Proof of Theorem \ref{t:varcons}]
\textbf{Step 1.}
We apply Lemmas \ref{l:stochFubini} and \ref{l:Aint=intA} to obtain the key equations for the proof of Theorem \ref{t:varcons}, equations \eqref{tvarcons_h11}
and \eqref{tvarcons_h12} below. As every adapted and measurable process with values in Polish space has a progressive version we may assume that $Y$ is
progressive. Consider the following process:
\begin{align*}
\Phi:[0,t]\times [0,t]\times \Omega\rightarrow \mL(H,F);\quad  \Phi(s,u,\omega) := 1_{u\leq s \leq t}T(t-s)B(Y(u)).
\end{align*}
Because $Y$ is strongly measurable, and because $B:F\rightarrow \mL(H,F)$ is continuous with respect to the strong operator topology and the semigroup $T(s)$ is
strongly continuous it follows that $\Phi$ is $H$-strongly measurable. Similarly, it follows from the fact that $Y$ is progressive that for all $s\in [0,t]$ and
all $h\in H$ the section $\Phi_s h$ is progressive. Thus conditions \eqref{meascond1} and \eqref{meascond2} of Lemma \ref{l:stochFubini} are satisfied. One
easily checks that condition \eqref{intcond1} of Lemma \ref{l:stochFubini} is satisfied by $\Phi$. Condition \eqref{intcond3} in Lemma \ref{l:stochFubini}
follows from assumption \eqref{si4}. Condition \eqref{intcond4} in Lemma \ref{l:stochFubini} follows from the definition of $\gamma(0,t;H,F)$, assumption
\eqref{si1} and the exponential boundedness of the semigroup: let $(h_k)_{k=1}^{n}$ be an arbitrary orthonormal sequence in $L^2(0,t;H)$, then
\begin{align*}
&\int_{0}^t\Big(\E\ \Big\n \sum_{k=1}^{n} \gamma_k  \int_{0}^t T(t-s)B(Y(u))h_k(u)1_{[0,s]}(u) du \Big\n_F^2\Big)^{\inv{2}}ds \\
& \qquad\qquad \leq  \int_{0}^t \n T(t-s) \n_{\mathcal{L}(F)} \Big( \E \ \Big\n \sum_{k=1}^n \gamma_k \int_{0}^t B(Y(u))h_k(u)1_{[0,s]}(u)du \Big\n_F^2
\Big)^{\inv{2}}ds\\
&\qquad\qquad\leq M_t \int_{0}^t \n B(Y)1_{[0,s])} \n_{\gamma(0,t;H,F)}ds \leq tM_t \n B(Y) \n_{\gamma(0,t;H,F)} < \infty,
\end{align*}
where $M_t:=\sup_{0\leq s\leq t}\n T(s) \n_{\mathcal{L}(F)}$ and we used the domination principle for Gaussian random variables to see that $\n B(Y)1_{[0,s]}
\n_{\gamma(0,t;H,F)}\leq \n B(Y) \n_{\gamma(0,t;H,F)}.$ Thus $\Phi$ satisfies all the conditions of the stochastic Fubini Lemma and we obtain:
\begin{align}
\int_{0}^{t} T(t-s) \int_{0}^{s} B(Y(u)) dW_H(u) ds & = \int_{0}^{t} \int_{u}^{t} T(t-s)B(Y(u)) dsdW_H(u) \  \textrm{a.s.} \label{tvarcons_h1a}
\end{align}
Observe that for all $h\in H$ one has $\int_{u}^{t} T(t-s)B(Y(u)) h ds \in D(A)$.
Hence by assumptions \eqref{si1} and \eqref{si2} we can apply Lemma \ref{l:Aint=intA} to obtain that the stochastic integral on the right-hand side of equation
\eqref{tvarcons_h1a} above is in $D(A)$ a.s., and we have:
\begin{align} 
A\int_{0}^{t} \int_{u}^{t} T(t-s)B(Y(u)) dsdW_H(u)&=\int_{0}^{t} A \int_{0}^{t-u} T(s)B(Y(u)) dsdW_H(u) \notag \\
& = \int_{0}^{t} (T(t-u)-I)B(Y(u))dW_H(u) \quad \textrm{a.s.} \label{tvarcons_h1b}
\end{align}
Combining equations \eqref{tvarcons_h1a} and \eqref{tvarcons_h1b} we obtain:
\begin{align}\label{tvarcons_h11}
&A\int_{0}^{t} T(t-s) \int_{0}^{s} B(Y(u)) dW_H(u) ds =\int_{0}^{t} (T(t-u)-I)B(Y(u))dW_H(u) \quad \textrm{a.s.}
\end{align}\par
Similarly using assumption \eqref{si5} one can prove that for $0\leq s \leq t$ the stochastic integrals in the equation below are well-defined and one has the following identity:
\begin{align}\label{tvarcons_h12}
&A\int_{0}^{t} \int_{0}^{s} T(s-u) B(Y(u)) dW_H(u) ds =\int_{0}^{t} (T(t-u)-I)B(Y(u))dW_H(u).
\end{align}
\textbf{Step 2.}
Assume $Y$ is a generalized strong solution to \eqref{SCP}, we prove that \eqref{varcon} holds. By \eqref{tvarcons_h11} and by the definition of a generalized
strong solution we have:
\begin{align*}
&Y(t)-Y_0-A\int_{0}^{t} Y(s) ds = \int_{0}^{t} B(Y(s)) dW_H(s)\\
&\qquad= \int_{0}^{t} T(t-s)B(Y(s)) dW_H(s) - A\int_{0}^{t} T(t-s) \int_{0}^{s} B(Y(u)) dW_H(u)ds.
\end{align*}
Let us consider the final term above without the $A$. By assumption and by Fubini's theorem one has:
\begin{align*}
&\int_{0}^{t} T(t-s) \int_{0}^{s} B(Y(u)) dW_H(u)ds\\
&\qquad = \int_{0}^{t} T(t-s) \left[Y(s)-Y_0-A\int_{0}^{s} Y(u) du\right] ds \\
&\qquad = \int_{0}^{t} T(t-s) Y(s) ds - \int_{0}^{t} T(t-s) Y_0 ds - A \int_{0}^{t} \int_{u}^{t} T(t-s)Y(u) dsdu\\
&\qquad = - \int_{0}^{t} T(t-s) Y_0 ds + \int_{0}^{t} Y(s) ds,
\end{align*}
which, when substituted to the earlier equation, gives:
\begin{align*}
&Y(t)-Y_0-A\int_{0}^{t} Y(s) ds \\
&\qquad = \int_{0}^{t} T(t-s)B(Y(s)) dW_H(s) +T(t)Y_0 - Y_0 - A\int_{0}^{t} Y(s)ds.
\end{align*}
On the other hand, if $Y$ satisfies \eqref{varcon}, then $\int_{0}^{t}Y(s)ds$ exists and is in $D(A)$ a.s.\ by \eqref{tvarcons_h12}, and therefore using this
equation we obtain:
\begin{align*}
& A\int_{0}^{t} Y(s) ds = A\int_{0}^{t} T(s)Y_0 ds + A\int_{0}^{t} \int_{0}^{s}T(s-u)B(Y(u))dW_H(u)ds\\
&\qquad = T(t)Y_0 - Y_0 + \int_{0}^{t} \left[T(t-u)-1\right] B(Y(u)) dW_H(u)\\
&\qquad = Y(t)-Y_0 -\int_{0}^{t} B(Y(u))dW_H(u).
\end{align*}
\end{proof}
Continuity of a process satisfying \eqref{varcon} can be proved by means of the factorization method as introduced in Section 2 of \cite{DaKwaZab87}. We give
the proof below; it is a straightforward adaptation of the proof of Theorem 3.3 in \cite{VerZim:08}.
\begin{thm}\label{t:reg}
Let $(T(t))_{t\geq 0}$ be a semigroup on a Banach space $F$ with the decoupling property. Let $Z:[0,t]\times \Omega\to \mL(H,F)$ be an $H$-strongly measurable
adapted process. Suppose that there exists $\alpha,p>0$, $\inv{p}<\alpha<\inv{2}$ and $M>0$ such that
\begin{align}\label{gbddin}
\sup_{0\leq s\leq t} \n u\mapsto (s-u)^{-\alpha} T(s-u) Z(u) \n_{L^p(\Omega;\gamma(0,s,H;F))} &\leq M.
\end{align}
Then the process
\begin{align*}
s\mapsto \int_{0}^{s} T(s-u) Z(u) dW_H(u)
\end{align*}
is well-defined and has a version with continuous paths. Moreover we have
$$ \E \ \sup_{0\leq s\leq t}\Big\n\int_{0}^{s} T(s-u) Z(u) dW_H(u)\Big\n_F^p<\infty. $$
\end{thm}
Before giving the proof of this theorem we mention the following corollary:
\begin{corol}\label{c:contSCP}
Consider the stochastic Cauchy problem \eqref{SCP} set in a Banach space $F$ that satisfies the decoupling property. The process $Y:[0,t]\times \Omega
\rightarrow F$ satisfying the variation of constants formula \eqref{varcon} belongs to $L^p(\Omega;C([0,t];F))$ if there exists $\alpha,p>0$,
$\inv{p}<\alpha<\frac{1}{2}$ such that
\begin{align*}
\sup_{0\leq s\leq t} \n u\mapsto (s-u)^{-\alpha} T(s-u) Y(u) \n_{L^p(\Omega;\gamma(0,s,H;F))} &<\infty.
\end{align*}
\end{corol}
\begin{proof}[Proof of Theorem \ref{t:reg}]
By assumption \eqref{gbddin} and Theorem \ref{t:stochint} it follows that for all $s\in [0,t]$ we can define
\begin{align*}
\Psi_1(s):=\int_{0}^{s} (s-u)^{-\alpha} T(s-u)Z(u) dW_H(u).
\end{align*}
By Proposition A.1 in \cite{vanNeervenVeraarWeis_SEEinUMD} the process $\Phi_1$ has a version which is adapted and strongly measurable. Moreover, by assumption
and inequality \eqref{BDG2} one has, for all $s\in [0,t]$,
\begin{align}\label{Lpest}
\E\ \n \Psi_1 (s)\n_F^p & \leq M,
\end{align}
whence $\Psi_1\in L^p(0,t;L^p(\Omega;F))$, and thus, by Fubini, $\Psi_1\in L^p(\Omega;L^p(0,t;F))$. Let $\Omega_0\subset \Omega$ denote the set on which
$\Psi_1\in L^p(0,t;F)$; we have $\P(\Omega_0)=1$.\par
By the domination principle for Gaussian random variables (see also \cite[Corollary 4.4]{vanNeervenVeraarWeis2}) it follows that for all $s\in [0,t]$ one has,
almost surely, 
$$ \n u\mapsto T(s-u) Z(u,\omega) \n_{\gamma(0,s,H;F)}\leq \n u\mapsto t^\alpha(s-u)^{-\alpha} T(s-u) Z(u,\omega) \n_{\gamma(0,s,H;F)} \quad \textrm{a.s.}$$
Thus by assumption we can define, for all $s\in [0,t]$,
\begin{align*}
\Psi_2(s) :=\int_{0}^{s} T(s-u)Z(u) dW_H(u),
\end{align*}
which again has a version that is adapted and strongly measurable.\par
It is proved in \cite{DaKwaZab87} that one may define a bounded operator $R_{\alpha}:L^p(0,t;F)\rightarrow C([0,t];F)$ by setting
\begin{align*}
(R_{\alpha}f)(s) & := \int_{0}^{s} (s-u)^{\alpha-1}T(s-u)f(u)du.
\end{align*}
Thus it remains to show that for almost all $\omega \in \Omega_0$ one has that for all $s\in [0,t]$ that
\begin{align}\label{fac}
\Psi_2(s) = \frac{\sin \pi \alpha}{\pi} (R_{\alpha}\Psi_1)(s),
\end{align}
i.e.\ that for all $x^*\in F^*$ one has
\begin{align*}
\langle \Psi_2(s), x^*\rangle & = \frac{\sin \pi \alpha}{\pi} \int_{0}^{s} (s-u)^{\alpha-1} \langle T(s-u)\Psi_1 (u), x^* \rangle du \quad \textrm{a.s.}
\end{align*}
This follows from a Fubini argument, see \cite[Theorem 3.5]{NeerVer} and \cite{DaKwaZab87}. The conditions necessary to apply the Fubini Theorem follow from the assumption \eqref{gbddin}. By \eqref{fac} and \eqref{Lpest} one has
$$  \E \sup_{0\leq s\leq t}\n\Psi_2(s)\n_F^p\leq C \E \int_0^t\n\Psi_1(s)\n_F^pds \leq tCM,$$
where $C$ is independent of $Z$. Thus the final estimate follows.
\end{proof}
\section{The Stochastic Delay Equation}\label{s:SDE}
\subsection{The variation of constants formula}
We now turn to the stochastic delay equation \eqref{SDE} as presented in the introduction and the related stochastic Cauchy problem \eqref{SDCP} on page \pageref{SDCP}. Recall that we assumed that \eqref{SDE} is set in a type 2 \textsc{umd} Banach space $E$ and that the related Cauchy problem is set in $\mathcal{E}^p(E)= E \times L^p(-1,0;E)$ for some $p\in [1,\infty)$. (The results in this article are also valid if $E$ is a type 2 Banach space with the decoupling property but we do not know of any such spaces that are not in fact \textsc{umd} spaces.)\par 
Recall that a Banach space $F$ is said to have {\em type $p\in [1,2]$} if there
exists a constant $C\ge 0$ such that for all
finite choices $x_1,\dots, x_k\in F$ we have
$$ \Big(\E\ \Big\n \sum_{j=1}^k \gamma_j x_j \Big\n_F^2\Big)^\frac12 \le C
\Big(\sum_{j=1}^k \n x_j\n_F^p\Big)^\frac1p,$$
where $(\gamma_j)_{j\geq 1}$ is a sequence of independent standard Gaussians. Hilbert spaces have type $2$ and $L^p$-spaces with $p\in [1,\infty)$ have type $\min\{p,2\}$. We refer to \cite{AlbKal} for more information, for our purposes we only need that in Banach spaces with type 2 the following embedding holds, see p.\ 1460 in \cite{vanNeervenVeraarWeis}:
\begin{align}\label{emb}
L^2(0,t;\gamma(H,F))\hookrightarrow \gamma(0,t;H,F).
\end{align}\par
Let $(\mathcal{T}(t))_{t\geq 0}$ denote the semigroup generated by $\mathcal{A}$, where $\mathcal{A}$ is the operator in \eqref{SDCP} defined by \eqref{opA} in the introduction. We define the projections $\pi_1:\mathcal{E}^p(E)\rightarrow E$ and $\pi_2 : \mathcal{E}^p(E) \rightarrow L^p(-1,0;E) $ as follows:
\begin{align*}
\pi_1 \left[\begin{array}{c} x \\ f \end{array}\right]&=x;& \pi_2 \left[\begin{array}{c} x \\ f \end{array}\right]&=f.
\end{align*} 
The following property of $(\mathcal{T}(t))_{t\geq 0}$ is intuitively obvious and useful in the following:
\begin{align}\label{propT}
\left(\pi_2 \mathcal{T}(t)\left[\begin{array}{c} x \\ f \end{array}\right]\right)(u) &= \pi_1 \mathcal{T}(t+u)\left[\begin{array}{c} x \\ f \end{array}\right].
\end{align}
for $f\in \mathcal{E}^p(E), u\in[-1,0], t>-u$ (for a proof see \cite{batkaiPiazzera}, Proposition 3.11). \par
The proof of the following lemma is straightforward and thus left to the reader:
\begin{lemma}\label{l:intex}
Let $t>0$, $p\in [1,\infty)$ and $x\in L^{p}(-1,t;E)$. Then the function $y:[0,t]\rightarrow L^{p}(-1,0;E)$, $y(s):=x_s$ is (Bochner) integrable and
\begin{align*}
\int_{0}^{t} y(s) ds & \in W^{1,p}(-1,0;E);\\
\left(\int_{0}^{t} y(s) ds\right)(u)&= \int_{0}^{t} x(s+u) ds \quad a.s.; &\left(\int_{0}^{t} y(s) ds\right)^{'} &=y(t)-y(0) \quad \textrm{a.s.}
\end{align*}
\end{lemma}
Generalized strong solutions to \eqref{SDCP} are equivalent to mild solutions:
\begin{thm}\label{t:Dvarcons}
Let $E$ be a type $2$ \textsc{umd} Banach space and let $p\in [1,\infty)$. Consider \eqref{SDCP}; i.e.\ let $\mathcal{A}$ defined by \eqref{opA} be the generator of the
$C_0$-semigroup $(\mathcal{T}(t))_{t\geq 0}$ on $\mathcal{E}^p(E)=E\times L^p(-1,0;E)$. Let $\mathcal{B}:\mathcal{E}^p(E)\rightarrow
\gamma(H,\mathcal{E}^p(E))$ be given by $\mathcal{B}([x,f]^{T})= [B([x,f]),0]^T$, where $B:\mathcal{E}^p(E) \rightarrow \gamma(H,E)$ is Lipschitz continuous.
Finally, let $W_H$ be an $H$-cylindrical Brownian motion adapted to $(\F_s)_{s\geq 0}$.\par
Let $Y:[0,\infty)\times \Omega \rightarrow \mathcal{E}^p(E)$ be a strongly measurable, adapted process satisfying 
\begin{align*}
\int_{0}^{t} \n Y(s)\n_{\mathcal{E}^p(E)}^{2} ds < \infty \quad \textrm{a.s.\ for all } t>0; 
\end{align*}
Then  $Y$ is a generalized strong solution to \eqref{SDCP} if and only if $Y$ is a solution to: 
\begin{align}\label{Dvarcon}
Y(t) &= \mathcal{T}(t)\vectwo{x_0}{f_0} + \int_{0}^{t} \mathcal{T}(t-s)\mathcal{B}(Y(s))dW_H(s) \quad \textrm{a.s.\ for all }t\geq 0.
\end{align}
\end{thm}
\begin{proof}
We apply Theorem \ref{t:varcons} to obtain the above assertion, for which we need to check condition \eqref{si5} and that the processes given by \eqref{si1}, \eqref{si2} and \eqref{si4} in that theorem are elements of $\gamma(0,t;H,\mathcal{E}^p(E))$ a.s.\ for all $t>0$. Let $t>0$ be fixed.\par
\textbf{Process \eqref{si1} in Theorem \ref{t:varcons}}. By the embedding \eqref{emb} and the Lipschitz-continuity of $\mathcal{B}$ we have:
\begin{align*}
\n s\mapsto \mathcal{B}(Y(s)) \n_{\gamma(0,t;H,\mathcal{E}^p(E))} & = \n s\mapsto B(Y(s)) \n_{\gamma(0,t;H,E)}\\
& \lesssim  \n s\mapsto B(Y(s)) \n_{L^2(0,t;\gamma(H,E))}\\
& \lesssim t^{\inv{2}}\n B(0)\n_{\gamma(H,E)} + K \n Y \n_{L^2(0,t;\mathcal{E}^p(E))},
\end{align*}
where $K$ is the Lipschitz-constant of $B$.\par
\textbf{Process \eqref{si2} in Theorem \ref{t:varcons}}.
By Lemma \ref{l:LpGammaFub} and embedding \eqref{emb} we have:
\begin{align*}
& \n u\mapsto \mathcal{T}(t-u)\mathcal{B}(Y(u)) \n_{\gamma(0,t;H,\mathcal{E}^p(E))} \\
& \qquad \qquad \qquad \qquad  \lesssim_p 
\n u\mapsto \pi_1 \mathcal{T}(t-u)\mathcal{B}(Y(u)) \n_{\gamma(0,t;H,E)} \\
& \qquad \qquad \qquad \qquad \quad + \n u\mapsto \pi_2 \mathcal{T}(t-u)\mathcal{B}(Y(u)) \n_{L^p(-1,0; \gamma(0,t;H,E))}\\
& \qquad \qquad \qquad \qquad \leq \n u\mapsto \pi_1 \mathcal{T}(t-u)\mathcal{B}(Y(u)) \n_{L^2(0,t;\gamma(H,E))} \\ & \qquad \qquad \qquad \qquad \quad + \n u\mapsto \pi_2 \mathcal{T}(t-u)\mathcal{B}(Y(u)) \n_{L^p(-1,0; L^2(0,t;\gamma(H,E))}.
\end{align*}
Set $M_t:=\sup_{u\in[0,t]}\n \mathcal{T}(u)\n_{\mL(\mathcal{E}^p(E))}$. By the ideal property of the $\gamma$-radonifying operators and the Lipschitz-continuity of $\mathcal{B}$ we have:
\begin{align*}
&\n u\mapsto \pi_1 \mathcal{T}(t-u)\mathcal{B}(Y(u)) \n_{L^2(0,t;\gamma(H,E))} \\
&\qquad \qquad \leq M_t \Big[ t^{\inv{2}}\n B(0)\n_{\gamma(H,E)} + K \n Y \n_{L^2(0,t;\mathcal{E}^p(E))}\Big],
\end{align*}
where $K$ is the Lipschitz-constant of $B$, and, by equality \eqref{propT},
\begin{align*}
& \n u\mapsto \pi_2 \mathcal{T}(t-u)\mathcal{B}(Y(u)) \n_{L^p(-1,0; L^2(0,t;\gamma(H,E))}  \\
& \qquad \qquad = \Big(\int_{-1}^{0} \n \pi_1 \mathcal{T}(t-u+s) \mathcal{B}(Y(u))\n^p_{L^2(0,t+s;\gamma(H,E)} ds \Big)^{\inv{p}} \\
& \qquad \qquad \leq M_t \Big[ t^{\inv{2}}\n B(0)\n_{\gamma(H,E)} + K\n Y \n_{L^2(0,t;\mathcal{E}^p(E))}\Big].
\end{align*}
\textbf{Process \eqref{si4} in Theorem \ref{t:varcons}}. Note that by Remark \ref{r:varcons} we may interpret $$\int_{0}^{t-u} \mathcal{T}(s)\mathcal{B}(Y(u)) ds$$ as a $\gamma(H,\mathcal{E}^p(E))$-valued Bochner integral. To prove that the process $$u\mapsto \int_{0}^{t-u} \mathcal{T}(s)\mathcal{B}(Y(u)) ds \in \gamma(0,t;H,\mathcal{E}^p(E))\quad \textrm{a.s.,}$$ observe that by Lemma \ref{l:LpGammaFub} and embedding \eqref{emb} we have:
\begin{align*}
& \Big\Vert u\mapsto \int_{0}^{t-u} \mathcal{T}(s)\mathcal{B}(Y(u))ds \Big\Vert_{\gamma(0,t;H,\mathcal{E}^p(E))} \\&\qquad \lesssim_p \Big\n u\mapsto \pi_1 \int_{0}^{t-u} \mathcal{T}(s)\mathcal{B}(Y(u))ds \Big\n_{L^2(0,t;\gamma(H,E))} \\ & \qquad \quad + \Big\n u\mapsto \pi_2  \int_{0}^{t-u} \mathcal{T}(s)\mathcal{B}(Y(u))ds \Big\n_{L^p(-1,0; L^2(0,t;\gamma(H,E))}.
\end{align*}
By Minkowski's integral inequality, the ideal property of the $\gamma$-radonifying operators and the Lipschitz-continuity of $\mathcal{B}$ we have:
\begin{align*}
&\Big\n u\mapsto \pi_1 \int_{0}^{t-u} \mathcal{T}(s)\mathcal{B}(Y(u))ds \Big\n_{L^2(0,t;\gamma(H,E))} \\
&\qquad \qquad \qquad \leq t M_t \Big[ t^{\inv{2}} \n B(0)\n_{\gamma(H,E)} + K \n Y \n_{L^2(0,t;\mathcal{E}^p(E))}\Big],
\end{align*}
and by equation \eqref{propT} and Lemma \ref{l:intex} we have:
\begin{align*}
& \Big\n u\mapsto \pi_2 \int_{0}^{t-u} \mathcal{T}(s)\mathcal{B}(Y(u))ds \Big\n_{L^p(-1,0; L^2(0,t;\gamma(H,E))}  \\
& \qquad \qquad = \Big( \int_{-1}^{0} \Big\n u\mapsto \pi_1 \int_{0}^{t-u+r} \mathcal{T}(s+r)\mathcal{B}(Y(u))ds \Big\n_{L^2(0,t;\gamma(H,E))}^p dr \Big)^{\inv{p}}\\
& \qquad \qquad \leq t M_t \Big[ t^{\inv{2}} \n B(0)\n_{\gamma(H,E)} + K\n Y \n_{L^2(0,t;\mathcal{E}^p(E))}\Big].
\end{align*}
\textbf{Condition \eqref{si5} in Theorem \ref{t:varcons}.}
From the estimates for process \eqref{si4} above we obtain:
\begin{align*}
& \int_0^t\n u\mapsto \mathcal{T}(s-u)\mathcal{B}(Y(u)) \n_{\gamma(0,s;H,\mathcal{E}^p(E))} ds\\
& \qquad \qquad \qquad \qquad \lesssim_p 
2tM_t\Big[ t^{\inv{2}}\n B(0)\n_{\gamma(H,E)} + K\n Y \n_{L^2(0,t;\mathcal{E}^p(E))}\Big].
\end{align*}
Having checked condition \eqref{si5} and that all processes are in $\gamma(0,t;H,E)$ a.s.\ we may apply Theorem \ref{t:varcons} to obtain the desired result.
\end{proof}
\begin{remark}\label{r:pisYinLp}
Let $p'$ be such that $\inv{p}+\inv{p'}=1$. By testing the stochastic convolution in equation \eqref{Dvarcon} against elements of $E^*\times L^{p'}(-1,0;E^*)$, which is norming for $\mathcal{E}^p(E)$, and applying equality \eqref{propT} one shows that
\begin{align*}
\int_{0}^{t} \pi_2 \mathcal{T}(t-s)\mathcal{B}(Y(s)) dW_H(s) &= u\mapsto \int_{0}^{t+u} \pi_1 \mathcal{T}(t-s+u)\mathcal{B}(Y(s))dW_H(s) \quad \textrm{a.s.}
\end{align*}
It thus follows from the variation of constants formula \eqref{voc} that if $Y$ is a generalized strong solution to \eqref{SDCP} then $\pi_2 Y(t)(u) = \pi_1 Y(t+u)$; in particular it follows that $\pi_1 Y\in L_{loc}^p(0,\infty;E)$ a.s.
\end{remark}

\subsection{Existence and uniqueness of the solution to (SDCP)}
We continue consider \eqref{SDCP} on page \pageref{SDCP}. Recall that $(\F_s)_{s\geq 0}$ is a filtration to which $W_H$ is adapted. For $t > 0$, $q\in [1,\infty)$ and $r \in [1,\infty]$ let $L^r_{\F}(0,t;L^q(\Omega; \mathcal{E}^p(E)))$ be the Banach space of $(\F_s)_{s\geq 0}$ adapted processes in $L^r(0,t;L^q(\Omega;\mathcal{E}^p(E)))$. In particular, $L^{\infty}_{\F}(0,t;L^q(\Omega; \mathcal{E}^p(E)))$ is the Banach space of $(\F_s)_{s\geq 0}$ adapted processes $Y$ such that
\begin{align*}
\n Y \n_{L^{\infty}_{\F}(0,t;L^q(\Omega; \mathcal{E}^p(E)))}& = \sup_{0\leq s \leq t} \big( \E\ \n Y(s)\n_{\mathcal{E}^p(E)}^q\big)^{\inv{q}} < \infty.
\end{align*}

\begin{thm}\label{t:exist}
Let the assumptions of Theorem \ref{t:Dvarcons} hold. In addition, assume that $Y_0:=[x_0,f_0]^T\in L^q(\F_0,\mathcal{E}^p(E))$ for some $q \in [2,\infty)$.
Then for every $t>0$ and every $r\in [2,\infty]$ there exists a unique process $Y\in L^r_{\F}(0,t;L^q(\Omega; \mathcal{E}^p(E)))$ for which \eqref{Dvarcon}
holds. In particular, this process is in $L^{\infty}_{\F}(0,t;L^q(\Omega; \mathcal{E}^p(E)))$.
\end{thm}
\begin{proof}
The final remark in the theorem follows from the existence of a solution in $L^{\infty}_{\F}(0,t;L^q(\Omega; \mathcal{E}^p(E)))$ and the uniqueness of the solution in $L^r_{\F}(0,t;L^q(\Omega; \mathcal{E}^p(E)))$.\par
Fix $r\in (2,\infty]$ and let $t>0$. Define $$L:L^r_{\F}(0,t;L^q(\Omega; \mathcal{E}^p(E)))\rightarrow L^r_{\F}(0,t;L^q(\Omega; \mathcal{E}^p(E)))$$ as follows:
\begin{align*}
L(Z)(s) := \mathcal{T}(s)Y_0 + \int_{0}^{s}\mathcal{T}(s-u)\mathcal{B}(Z(u))dW_H(u),
\end{align*}
where $s\in [0,t]$. Set $M_t:= \sup_{0\leq u\leq t}\n \mathcal{T}(u) \n_{\mathcal{L}(\mathcal{E}^p(E))}$. To prove that $L(Z)$ is indeed in $L^r_{\F}(0,t;L^q(\Omega; \mathcal{E}^p(E)))$, first observe that by inequality \eqref{BDG2} and the proof of Theorem \ref{t:Dvarcons}, we have:
\begin{align*}
&(\E\ \n L(Z)(s) \n_{\mathcal{E}^p(E)}^q)^{\inv{q}}\\
&\lesssim_q (\E\ \n \mathcal{T}(s)Y_0\n_{\mathcal{E}^p(E)}^q)^{\inv{q}} + \n u \mapsto \mathcal{T}(s-u)\mathcal{B}(Z(s)) \n_{L^q(\Omega,\gamma(0,s;H,\mathcal{E}^p(E)))} \\
&\leq M_t \Big[(\E \ \n Y_0\n_{\mathcal{E}^p(E)}^q)^{\inv{q}} + \Big(\E\ \Big[ s^{\inv{2}}\n B(0)\n_{\gamma(H,E)} + K \Big( \int_{0}^{s} \n Z(u) \n_{\mathcal{E}^p(E)}^2 du \Big)^{\inv{2}} \Big]^q\Big)^{\inv{q}}\Big],
\end{align*}
and thus from Minkowski's integral inequality, the H\"older inequality and the fact that $r\geq q \geq 2$ we obtain:
\begin{align*}
&(\E\ \n L(Z)(s) \n_{\mathcal{E}^p(E)}^q)^{\inv{q}}\\
&\leq M_t \Big[(\E \ \n Y_0\n_{\mathcal{E}^p(E)}^q)^{\inv{q}}+t^{\inv{2}}\n B(0)\n_{\gamma(H,E)}+K\Big(\int_0^s \big[\E\ \n Z(u)\n_{\mathcal{E}^p(E)}^q\big]^{\frac{2}{q}} du \Big)^\inv{2}\Big]\\
&\leq M_t \Big[(\E \ \n Y_0\n_{\mathcal{E}^p(E)}^q)^{\inv{q}}+t^{\inv{2}}\n B(0)\n_{\gamma(H,E)}+Ks^{\inv{2}-\frac{1}{r}}\n Z \n_{L^r(0,t;L^q(\Omega;\mathcal{E}^p(E)))}\Big],
\end{align*}
for every $s\in[0,t]$, where $K$ is the Lipschitz constant of $B$. (In the case $r=\infty$ we interpret $\frac{1}{r}=0$.) Taking $r^{\textrm{th}}$ powers in the above and integrating with respect to $s$ gives that $L(Z) \in  L^r_{\F}(0,t;L^q(\Omega; \mathcal{E}^p(E)))$. In the same way as the above estimate, one has for $Z_1,Z_2\in L^r_{\F}(0,t;L^q(\Omega; \mathcal{E}^p(E)))$:
\begin{align*}
\n L(Z_1) - L(Z_2) \n_{L^r_{\F}(0,t;L^q(\Omega; \mathcal{E}^p(E)))} & \lesssim_q K t^{\inv{2}} M_t \n Z_1 - Z_2 \n_{L^r_{\F}(0,t;L^q(\Omega;
\mathcal{E}^p(E)))},
\end{align*}
so this is a strict contraction for $t$ sufficiently small. Hence by the Banach fixed-point theorem there exists a unique $Y\in
L^r_{\F}(0,t;L^q(\Omega; \mathcal{E}^p(E)))$ that satisfies \eqref{voc}. By repeating this argument one obtains a solution for arbitrary $t>0$.
\end{proof}
\subsection{Continuity of the solution to (SDCP)}
\begin{thm}\label{t:regSDCP}
Let the assumptions of Theorem \ref{t:Dvarcons} hold. In addition, assume that $Y_0:=[x_0,f_0]^T\in L^q(\F_0,\mathcal{E}^p(E))$ for some $q\in (2,\infty)$.
Let $t>0$. Then the solution $Y\in L^\infty_{\F}(0,t;L^q(\Omega; \mathcal{E}^p(E)))$ to \eqref{SDCP} as given by Theorem \ref{t:exist} satisfies $Y \in
L^q(\Omega;C([0,t]; \mathcal{E}^p(E)))$.
\end{thm}
\begin{proof}
The statement follows from Corollary \ref{c:contSCP} if it holds that for some $\alpha\in (\inv{q},\inv{2})$ we have:
\begin{align}\label{supcond}
\sup_{0\leq s\leq t} \n u\mapsto (s-u)^{-\alpha} \mathcal{T}(s-u)\mathcal{B}(Y(u)) \n_{L^q(\Omega,\gamma(0,s;\mathcal{E}^p(E)))} < \infty.
\end{align}
Fix $\alpha\in (\inv{q},\inv{2})$ and $s\in [0,t]$. By Lemma \ref{l:LpGammaFub} and embedding \eqref{emb} we have:
\begin{equation}
\begin{aligned}\label{splitup}
&\n u\mapsto (s-u)^{-\alpha} \mathcal{T}(s-u)\mathcal{B}(Y(u)) \n_{\gamma(0,s;\mathcal{E}^p(E))} \\
&\qquad \qquad \lesssim_{p} \n u\mapsto \pi_1 (s-u)^{-\alpha} \mathcal{T}(s-u)\mathcal{B}(Y(u)) \n_{L^2(0,s;\gamma(H,E))} \\ 
&\qquad \qquad \quad + \n u\mapsto  (s-u)^{-\alpha} \pi_2\mathcal{T}(s-u)\mathcal{B}(Y(u)) \n_{L^p(-1,0; L^2(0,s;\gamma(H,E))},
\end{aligned}
\end{equation}
where $M_t:=\sup_{u\in[0,t]}\n \mathcal{T}(u)\n_{\mL(\mathcal{E}^p(E))}$. Concerning the final term in \eqref{splitup}; by \eqref{propT} and by the ideal property of the $\gamma$-radonifying operators we have:
\begin{align*}
& \n u\mapsto (s-u)^{-\alpha} \pi_2\mathcal{T}(s-u)\mathcal{B}(Y(u)) \n_{L^p(-1,0; L^2(0,s;\gamma(H,E))} \\
&\qquad= \Big[ \int_{-1}^{0} \Big(\int_{0}^{s+r}   (s-u)^{-2\alpha}\n \pi_1 \mathcal{T}(s-u+r)\mathcal{B}(Y(u)) \n_{\gamma(H,E)}^2 du \Big)^{\frac{p}{2}} dr \Big]^{\inv{p}}\\
& \qquad \leq M_t \Big[ \int_{-1}^{0} \Big(\int_{0}^{s}   (s-u)^{-2\alpha}\n B(Y(u)) \n_{\gamma(H,E)}^2 du \Big)^{\frac{p}{2}} dr \Big]^{\inv{p}}\\
& \qquad = M_t \Big( \int_{0}^{s}  (s-u)^{-2\alpha}\n B(Y(u)) \n_{\gamma(H,E)}^2 du \Big)^{\frac{1}{2}}.
\end{align*}
As $q>2$, and using in addition the Lipschitz-continuity of $B$, it follows that:
\begin{align*}
& \n u\mapsto (s-u)^{-\alpha} \pi_2\mathcal{T}(s-u)\mathcal{B}(Y(u)) \n_{L^q(\Omega;L^p(-1,0; L^2(0,s;\gamma(H,E)))} \\
&\qquad \leq M_t \Big( \int_{0}^{s}  (s-u)^{-2\alpha}\big[\E\n B(Y(u)) \n_{\gamma(H,E)}^q\big]^{\frac{2}{q}} du \Big)^{\frac{1}{2}}\\
&\qquad \leq (1-2\alpha)^{-\inv{2}} M_t s^{\inv{2}-\alpha} \big[\n B(0)\n_{\gamma(H,E)} + K\sup_{u\in[0,s]} (\E \ \n Y(u) \n_{\mathcal{E}^p(E)}^{q})^{\inv{q}}\big] < \infty,
\end{align*}
where $K$ is the Lipschitz constant of $B$. The estimate for the first term on the right-hand side of \eqref{splitup} is similar, but slightly simpler; one obtains:
\begin{align*}
& \n u\mapsto (s-u)^{-\alpha} \pi_1\mathcal{T}(s-u)\mathcal{B}(Y(u)) \n_{L^q(\Omega; L^2(0,s;\gamma(H,E))} \\
&\qquad \leq (1-2\alpha)^{-\inv{2}} M_t s^{\inv{2}-\alpha} \big[\n B(0)\n_{\gamma(H,E)} + K\sup_{u\in[0,s]} (\E \ \n Y(u) \n_{\mathcal{E}^p(E)}^{q})^{\inv{q}}\big] < \infty.
\end{align*}
From the above estimates and the fact that $s^{\inv{2}-\alpha}\leq t^{\inv2-\alpha}$ because $\alpha<\tfrac{1}{2}$, we conclude that \eqref{supcond} holds.
\end{proof}

\subsection{Equivalence of solutions to (SDE) and (SDCP)}
Consider the problem \eqref{SDE} as given in the introduction with a fixed $p\in [1,\infty)$. 
\begin{defn}\label{d:strongsol_SDE}
A process $X:[-1,\infty)\times\Omega \rightarrow E$ is called a \emph{strong solution} to \eqref{SDE} if it is measurable and adapted to $(\F_t)_{t\geq 0}$  and for all $t\geq 0$ one has:
\begin{enumerate}
\item\label{strsol_1} $ \int_{0}^{t}  |X (s)|^{2\maxsym p} ds< \infty$ a.s.;
\item $X|_{[-1,0)}=f_0$,
\item $\int_{0}^{t} X(s)ds \in D(A)$ for all $t>0$ a.s.;
\end{enumerate}
and
\begin{align} \label{SDE_int}
X(t)-x_0&= A\int_{0}^{t}X(s)ds+ C \int_{0}^{t} X_s ds + \int_0^t B(X(s),X_s)dW_H(s) \quad \textrm{a.s.}
\end{align}
\end{defn}
\begin{remark}\label{r:intinWp}
Note that by condition \eqref{strsol_1} and Lemma \ref{l:intex} one has $\int_{0}^{t} X_s ds\in W^{1,p}(-1,0;E)$ a.s. Moreover, for any $t>0$; by Minkowski's integral inequality 
\begin{align*}
\Big(\int_{0}^{t} \n X_s \n_{L^p(E)}^2 ds\Big)^{\inv{2}} & = \Big(\int_{0}^{t} \Big[ \int_{s-1}^{s} \n X(u) \n_{E}^p du \Big]^{\frac{2}{p}} ds\Big)^{\inv{2}}\\
&\leq \n f_0 \n_{L^p} + \Big(\int_{0}^{t} \Big[ \int_{0}^{t} \n X(u)\n_E^{p\maxsym 2} du\Big]^{\frac{p\minsym 2}{p\maxsym 2}}ds \Big)^{\inv{p\minsym 2}}\\
& = \n f_0 \n_{L^p} + t^{\inv{p\minsym 2}}\Big[ \int_{0}^{t} \n X(u)\n_E^{p\maxsym 2} du\Big]^{\frac{1}{p\maxsym 2}}<\infty \quad \textrm{a.s.}
\end{align*}
Hence by condition \eqref{strsol_1} the stochastic integral on right hand side of \eqref{SDE_int} is well defined.
\end{remark}
\begin{thm}\label{t:repr}
\begin{enumerate}
\item Let $X$ be a strong solution to \eqref{SDE}, then the process $Y$ defined by $Y(t):=[X(t),X_t]^T$ is a generalized strong solution to \eqref{SDCP}. \label{repr1}
\item On the other hand, if $Y$ is a generalized strong solution to \eqref{SDCP} then the process defined by $X|_{[-1,0)}=f_0$, $X(t):=\pi_1(Y(t))$ for $t\geq 0$ is a strong solution to \eqref{SDE}. \label{repr2}
\end{enumerate}
\end{thm}
\begin{proof}
Part \eqref{repr1}. In the proof of Theorem \ref{t:Dvarcons} we saw that $s\mapsto \mathcal{B}(Y(s))$ is stochastically integrable if $Y\in L^2(0,t; \mathcal{E}^p(E))$ a.s., which follows from Definition \ref{d:strongsol_SDE}, by Remark \ref{r:intinWp}. From Lemma \ref{l:intex} above it follows that $Y$ is integrable a.s.:
\begin{align*}
\int_0^t Y(s) ds &= \vectwo{\int_{0}^{t} X(s) ds }{\int_{0}^{t}X_s ds}\quad \textrm{a.s.}
\end{align*}
and that $\int_{0}^{t}X_s ds \in W^{1,p}(-1,0;E)$ a.s.\ and $\int_{0}^{t}X_s ds (0)= \int_{0}^{t}X(s) ds \in D(A)$. Hence $\int_0^t Y(s) ds \in D(\mathcal{A})$ a.s.\ and again by Lemma \ref{l:intex} and by assumption we have, a.s.:
\begin{align*}
\mathcal{A} \int_0^t Y(s) ds &= \vectwo{A\int_{0}^t X(s)ds + C\int_{0}^{t} X_s ds}{X_t - f_0} \\&= \vectwo{X(t)-x_0-\int_{0}^{t}B(X(s))dW_H(s)}{X_t-f_0}.
\end{align*}
Combining this equality with the following:
\begin{align*}
\int_{0}^{t} \mathcal{B}(Y(s)) dW_H(s) &= \int_{0}^{t} \vectwo{B(Y(s),Y_s)}{0} dW_H(s) = \vectwo{\int_{0}^{t}B(X(s),X_s)dW_H(s)}{0}
\end{align*}
we see $Y$ satisfies Definition \ref{d:strongsolSCP}.\par
Part \eqref{repr2}. Let $Y$ be a generalized strong solution to \eqref{SDCP} and define $X|_{[-1,0)}=f_0$, $X(t):=\pi_1(Y(t))$ for $t\geq 0$. Recall from Remark \ref{r:pisYinLp} that $\pi_2 Y(t)= u\mapsto \pi_1 Y(s+u)=X_s$. Thus from the definitions of a generalized strong solution and from the generator $\mathcal{A}$ we obtain
$$ X(s) - x_0 = A\int_{0}^t X(s)ds + C \int_{0}^{t} X_sds + \int_{0}^{t} B(X(s),X_s) dW_H(s)\quad a.s.$$
\end{proof}\par
\begin{corol}\label{r:voc2}
$X$ is a strong solution to \eqref{SDE} if and only if $X$ satisfies
\begin{align*}
X(t) = \pi_1\mathcal{T}(t)\vectwo{x_0}{f_0} +\int_{0}^{t}\pi_1 \mathcal{T}(t-s) B(X(s))dW_H(s) \quad\textrm{a.s.}
\end{align*}
\end{corol}
From Theorem \ref{t:exist} and Theorem \ref{t:repr} we obtain:
\begin{corol}\label{c:SDEexist}
Consider \eqref{SDE}. Assume $x_0\in L^q(\F_0;E)$ and $f_0\in L^q(\F_0;L^p)$ for some $p\in [1,\infty)$, $q\in [2,\infty)$. Then \eqref{SDE} has a unique
strong solution in $L^r(0,t;L^q(\Omega;E))$ for every $r\in [2,\infty]$ and every $t>0$.
\end{corol}
Combining Theorem \ref{t:regSDCP} and Theorem \ref{t:repr} we obtain:
\begin{corol}\label{c:SDEcont}
Consider \eqref{SDE}. Assume $x_0\in L^q(\F_0;E)$ and $f_0\in L^q(\F_0;L^p)$ for some $p\in [1,\infty)$, $q\in (2,\infty)$. The strong solution $X\in
L^\infty(0,t;L^q(\Omega;E))$ to \eqref{SDE} given by Corollary \ref{c:SDEexist} satisfies $X\in L^q(\Omega;C([0,t];E))$. 
\end{corol}

\begin{remark}\label{r:strongDPZ} One cannot hope to obtain a strong solution to \eqref{SDCP} as defined in the monograph of Da Prato and Zabczyk \cite{DaPratoZabczyk}, i.e.\ a process $Y$ such that $Y(t)\in D(\mathcal{A})$ a.s.\ for all $t\geq 0$ and
\begin{align*}
Y(t)-\vectwo{x_0}{f_0}=\int_{0}^{t}\mathcal{A}Y(s)ds + \int_{0}^{t}\mathcal{B}(Y(s))dW_H(s)\quad \textrm{a.s.\ for all }t\geq 0,
\end{align*}
unless the problem is deterministic, because of the following:
\begin{prop}
Let $E=\R$. If a generalized strong solution $Y$ to \eqref{SDCP} satisfies $Y(s)\in D(\mathcal{A})$ a.s.\ for all $s\in[0,t]$ then $\mathcal{T}(s)[x_0,f_0]^T\in Null(\mathcal{B})$ and $Y(s)=\mathcal{T}(s)[x_0,f_0]^T$ a.s.\ for almost all $s\in [0,t]$, i.e. \eqref{SDCP} is deterministic.
\end{prop}
\begin{proof}
Define $X:=\pi_1(Y)$, then $X$ is a generalized strong solution to $\eqref{SDE}$ by Theorem \ref{t:repr}. If $Y(s)\in D(\mathcal{A})$ for all $s\in[0,t]$ a.s.\ then $X\in W^{1,p}(0,t)$ a.s., i.e.\ by Lemma \ref{l:intex} the process $I(\mathcal{B}(Y)):[0,t]\times \Omega \rightarrow \mathbb{R}$ defined by $I(\mathcal{B}(Y))(s)= \int_{0}^{s}\mathcal{B}(Y(u)) dW_H(u)$ is in $W^{1,p}(0,t)$ a.s. Recall that the quadratic variation of $I(\mathcal{B}(Y))$ is given by 
\begin{align*}
V^2_t(I(\mathcal{B}(Y)) &= \int_0^{t} \mathcal{B}^2(Y(s))ds,
\end{align*}
and hence by Problem 1.5.11 in \cite{KaratzasShreve}  the process $I(\mathcal{B}(Y))$ can only be of bounded variation (and hence only possibly in $W^{1,p}(0,t)$) on the set
\begin{align*}
&\left\{\omega\in\Omega\,:\, \int_{0}^{t}\mathcal{B}^2(Y(s,\omega)) ds =0 \right\}\\
&\qquad =\left\{\omega\in\Omega\,:\, Y(s,\omega)\in Null(\mathcal{B})\textrm{ for almost all }s\in [0,t] \right\}.
\end{align*}
Thus if $I(\mathcal{B}^2(Y(s)))$ is to be in $W^{1,p}(0,t)$ a.s.\ then one has
\begin{align*}
Y(s) - \vectwo{x_0}{f_0} &= \mathcal{A}\int_{0}^{s} Y(u) du \quad \textrm{a.s.\ for all }s\in[0,t],
\end{align*}
which implies that $Y(s)=\mathcal{T}(s)[x_0,f_0]^T$ and $\mathcal{T}(s)[x_0,f_0]^T\in Null(\mathcal{B})$ a.s.\ for all $s\in [0,t]$.
\end{proof}
\end{remark}
\begin{remark} \label{r:invmeasure}
We can use Theorem \ref{t:repr} to find a stationary solution to \eqref{SDE} with additive noise, i.e. $B(Y(s))=b\in \gamma(H,E)$. It follows from Proposition 4.4 in \cite{vanNeervenWeis_invmeasure} that in this case \eqref{SDCP} admits invariant measure if and only if the function $$t\mapsto \mathcal{T}(t)[b,0]^T$$ represents an element of 
$\gamma(0,\infty;H,\mathcal{E}^p(E))$. By Lemma \ref{l:LpGammaFub}, embedding \eqref{emb} and equality \eqref{propT}
 this is the case if $\pi_1 \mathcal{T}(t)[b,0]^T \in L^2(0,\infty;\gamma(H,E))$, i.e.\ in particular if $(\mathcal{T}(t))_{t\geq 0}$ is exponentially stable.
\end{remark}
\section*{Acknowledgments}
The authors wish to thank Anna Chojnowska-Michalik, Jan van Neerven and Mark Veraar for helpful comments.


\begin{thebibliography}{10}

\bibitem{AlbKal}
F.~Albiac and N.J. Kalton.
\newblock {\em \rm ``{T}opics in {B}anach {S}pace {T}heory''}, volume 233 of
  {\em Graduate Texts in Mathematics}.
\newblock Springer, New York, 2006.

\bibitem{batkaiPiazzera}
A.~B{\'a}tkai and S.~Piazzera.
\newblock {\em Semigroups for delay equations}, volume~10 of {\em Research
  Notes in Mathematics}.
\newblock A K Peters Ltd., Wellesley, MA, 2005.

\bibitem{BierGaansVer09}
J.~Bierkens, O.~van Gaans, and S.V. Lunel.
\newblock Existence of an invariant measure for stochastic evolutions driven by
  an eventually compact semigroup.
\newblock {\em J. Evol. Equ.}, 9(4):771--786, 2009.

\bibitem{brze95}
Z.~Brze{\'z}niak.
\newblock Stochastic partial differential equations in {M}-type {$2$} {B}anach
  spaces.
\newblock {\em Potential Anal.}, 4(1):1--45, 1995.

\bibitem{chojnowskaMichalik_reprSDE}
A.~Chojnowska-Michalik.
\newblock Representation theorem for general stochastic delay equations.
\newblock {\em Bull. Acad. Polon. Sci. S\'er. Sci. Math. Astronom. Phys.},
  26(7):635--642, 1978.

\bibitem{coxveraar}
S.G. Cox and M.C. Veraar.
\newblock Vector-valued decoupling and the {B}urkholder-{D}avis-{G}undy
  inequality.
\newblock {\em Available at \texttt{http://fa.its.tudelft.nl/\~{}veraar/}}.

\bibitem{Crewe:10}
P.~Crewe.
\newblock Infinitely delayed stochastic evolution equations in {UMD} {B}anach
  spaces.
\newblock arXiv:1011.2615v1

\bibitem{DaKwaZab87}
G.~Da~Prato, S.~Kwapie{\'n}, and J.~Zabczyk.
\newblock Regularity of solutions of linear stochastic equations in {H}ilbert
  spaces.
\newblock {\em Stochastics}, 23(1):1--23, 1987.

\bibitem{DaPratoZabczyk}
G.~Da~Prato and J.~Zabczyk.
\newblock {\em Stochastic equations in infinite dimensions}, volume~44 of {\em
  Encyclopedia of Mathematics and its Applications}.
\newblock Cambridge University Press, Cambridge, 1992.

\bibitem{HofJor}
J.~Hoffmann-J{\o}rgensen.
\newblock Sums of independent {B}anach space valued random variables.
\newblock {\em Studia Math.}, 52:159--186, 1974.

\bibitem{KaratzasShreve}
I.~Karatzas and S.E. Shreve.
\newblock {\em Brownian motion and stochastic calculus}, volume 113 of {\em
  Graduate Texts in Mathematics}.
\newblock Springer-Verlag, New York, second edition, 1991.

\bibitem{Kwa}
S.~Kwapie{\'n}.
\newblock On {B}anach spaces containing {$c\sb{0}$}.
\newblock volume~52, pages 187--188. 1974.
\newblock A supplement to the paper by J. Hoffmann-J{\o}rgensen: ``Sums of
  independent Banach space valued random variables'' (Studia Math. {\bf 52}
  (1974), 159--186).

\bibitem{KwWo1}
S.~Kwapie{\'n} and W.A. Woyczy{\'n}ski.
\newblock Tangent sequences of random variables: basic inequalities and their
  applications.
\newblock In {\em Almost everywhere convergence (Columbus, OH, 1988)}, pages
  237--265. Academic Press, Boston, MA, 1989.

\bibitem{Liu:08}
K.~Liu.
\newblock Stochastic retarded evolution equations: {G}reen operators,
  convolutions, and solutions.
\newblock {\em Stoch. Anal. Appl.}, 26(3):624--650, 2008.

\bibitem{Mao:97}
X.~Mao.
\newblock {\em Stochastic differential equations and their applications}.
\newblock Horwood Publishing Series in Mathematics \& Applications. Horwood
  Publishing Limited, Chichester, 1997.

\bibitem{Moh:84}
S.E.A. Mohammed.
\newblock {\em Stochastic functional differential equations}, volume~99 of {\em
  Research Notes in Mathematics}.
\newblock Pitman (Advanced Publishing Program), Boston, MA, 1984.

\bibitem{vanNeervenRiedle_SDE}
J.M.A.M.~van Neerven and M.~Riedle.
\newblock A semigroup approach to stochastic delay equations in spaces of
  continuous functions.
\newblock {\em Semigroup Forum}, 74(2):227--239, 2007.

\bibitem{NeerVer}
J.M.A.M.~van Neerven and M.C. Veraar.
\newblock On the stochastic {F}ubini theorem in infinite dimensions.
\newblock In {\em Stochastic partial differential equations and
  applications---{VII}}, volume 245 of {\em Lect. Notes Pure Appl. Math.},
  pages 323--336. Chapman \& Hall/CRC, Boca Raton, FL, 2006.

\bibitem{vanNeervenVeraarWeis}
J.M.A.M.~van Neerven, M.C. Veraar, and L.~Weis.
\newblock Stochastic integration in {UMD} {B}anach spaces.
\newblock {\em Ann. Probab.}, 35(4):1438--1478, 2007.

\bibitem{vanNeervenVeraarWeis_SEEinUMD}
J.M.A.M.~van Neerven, M.C. Veraar, and L.~Weis.
\newblock Stochastic evolution equations in {UMD} {B}anach spaces.
\newblock {\em J. Funct. Anal.}, 255:940--993, 2008.

\bibitem{vanNeervenWeis_invmeasure}
J.M.A.M.~van Neerven and L.~Weis.
\newblock Invariant measures for the linear stochastic {C}auchy problem and
  {$R$}-boundedness of the resolvent.
\newblock {\em J. Evol. Equ.}, 6(2):205--228, 2006.

\bibitem{riedle_SDE}
M.~Riedle.
\newblock Solutions of affine stochastic functional differential equations in
  the state space.
\newblock {\em J. Evol. Equ.}, 8(1):71--97, 2008.

\bibitem{TanLiuTru:02}
T.~Taniguchi, K.~Liu, and A.~Truman.
\newblock Existence, uniqueness, and asymptotic behavior of mild solutions to
  stochastic functional differential equations in {H}ilbert spaces.
\newblock {\em J. Differential Equations}, 181(1):72--91, 2002.

\bibitem{vanNeervenVeraarWeis2}
J.~M. A.~M. van Neerven and L.~Weis.
\newblock Stochastic integration of functions with values in a {B}anach space.
\newblock {\em Studia Math.}, 166(2):131--170, 2005.

\bibitem{VerZim:08}
M.~Veraar and J.~Zimmerschied.
\newblock Non-autonomous stochastic {C}auchy problems in {B}anach spaces.
\newblock {\em Studia Math.}, 185(1):1--34, 2008.

\end{thebibliography}
\end{document}